\documentclass[12pt]{article}
\usepackage{amsmath,amsthm,amssymb}
\usepackage{graphicx}

\newcommand{\C}{\mbox{\rm \,l\kern-0.52em C}}
\newcommand{\Ce}{\rm \,l\kern-0.35em C}
\newcommand{\R}{{\rm l}\!{\rm R}}
\newcommand{\Z}{{\sf Z}\!\!{\sf Z}}
\newcommand{\Q}{\,{\rm l}\!\!\!{\rm Q}}

\newcommand{\N}{{\rm l}\!{\rm N}}

\newcommand{\Sup}{{\rm Sup}}

\newtheorem{theorem}{Theorem}[section]
\newtheorem{deflem}[theorem]{Definition and Lemma}
\newtheorem{definition}[theorem]{Definition}
\newtheorem{prop}[theorem]{Proposition}

\newtheorem{cor}[theorem]{Corollary}

\newtheorem{lemma}[theorem]{Lemma}
\newtheorem{remark}[theorem]{Remark}

\renewenvironment{proof}{{\bf Proof:}}{\mbox{}\hfill $\Box$}
\newenvironment{proof1}{{\bf Proof of Theorem 6.1:}}{\mbox{}\hfill $\Box$}
\newenvironment{proof2}{{\bf Proof sketch:}}{\mbox{}\hfill $\Box$}
\theoremstyle{definition}
\date{}

\title{The Chern-Connes character is not rationally injective}

\author{Michael Puschnigg}
\date{}

\begin{document}
\maketitle

{\bf Abstract:}
We show that the Chern-Connes character from Kasparov's bivariant K-theory to bivariant local cyclic cohomology is not always rationally injective. 
Counterexamples are provided by the reduced group $C^*$-algebras of word-hyperbolic groups with Kazhdan's property (T). The proof makes essential use of Skandalis' work on K-nuclearity and of Lafforgue's recent demonstration of the Baum-Connes conjecture with coefficients for word-hyperbolic groups.\\
\\

{\bf Keywords}: Kasparov theory, local cyclic cohomology, Chern-Connes character, word-hyperbolic groups, Property (T), Baum-Connes conjecture.\\
\\

{\bf AMS Classification:} 19K35, 46L80, 22D25.

\section{Introduction}
The classical Chern character is the (essentially unique) multiplicative 
natural transformation 
$$
ch: \, K^*_{top}(X) \,\longrightarrow H^*(X,\Q)
\eqno(1.1)
$$
from the Bott-periodic Atiyah-Hirzebruch topological $K$-theory of a locally compact Hausdorff space $X$ to its cohomology with coefficients in the constant sheaf $\Q$.\\
\\
It plays a central role in $K$-theory. Grothendieck introduced $K$-theory in order to derive his version of the Riemann-Roch Theorem, which compares Poincar\'e duality in $K$-theory and rational cohomology via the Chern character \cite{BS}.
In his work on index theory, Atiyah realized that elliptic operators on a compact manifold define cycles in the topological $K$-homology of the manifold. 
From this point of view the famous Atiyah-Singer Index theorem becomes a statement about the pairing between $K$-theory and $K$-homology \cite{At}. The Chern character plays again a central role because it transforms the index theorem into the index formula which allows explicit calculations \cite{AS}.\\
\\
The Chern character on $K$-theory is an isomorphism rationally
$$
ch\otimes\Q:\,K^*_{top}\otimes\Q\overset{\simeq}{\longrightarrow}H^*(-,\Q),
\eqno(1.2)
$$
so that not too much information is lost under this transformation. This fact makes $K$-theory to some extent computable because ordinary cohomology can be calculated by standard techniques from homological algebra.\\
\\
The natural framework for index theory is the Bott-periodic operator $K$-theory of Banach algebras and in particular Kasparov's bivariant $KK$-theory of $C^*$-algebras. The central tool of the latter theory is the Kasparov product which allows to treat generalized index theory from a very conceptual point of view. Moreover, bivariant $K$-theory can (contrary to ordinary $K$-theory or $K$-homology) be characterized by a simple system of axioms. On the subcategory of commutative $C^*$-algebras, which is anti-equivalent to the category of locally compact Hausdorff spaces, the functors obtained from $KK$ by fixing one of the two variables coincide with topological $K$-theory resp. $K$-homology.\\
\\
It was the search for a generalized Chern character in operator $K$-theory which led A.~Connes to the invention of cyclic (co)homology \cite{Co1}. He defined a cyclic Chern character 
$$
ch: \, K_*(-) \,\longrightarrow HP_*(-)
\eqno(1.3)
$$
from operator $K$-theory to his periodic cyclic homology. As for ordinary cohomology of spaces, the cyclic homology of an associative (Banach) algebra turns out to be computable by methods of homological algebra.
If $A={\cal C}^\infty(M)$ is the algebra of smooth functions on a compact manifold, then its periodic cyclic homology coincides with the de Rham cohomology of $M$, and Connes' cyclic Chern character coincides with the ordinary one \cite{Co1}.\\
\\
Due to its "noncommutative nature", cyclic cohomology serves also as a target of Chern characters in $K$-homology. In \cite{Co1} Connes gave explicit formulas for the character of finitely summable Fredholm modules over a Banach algebra (these generalize the notion of elliptic operators on compact manifolds) with values in periodic cyclic cohomology. In \cite{Co2} he obtained a much more general character formula for the huge class of Theta-summable Fredholm modules. It takes values in his entire cyclic cohomology.\\
\\
The first attempt to construct a character on bivariant $K$-theory was made by Nistor \cite{Ni}, who attached well behaved bivariant character cocycles to finitely summable Kasparov bimodules. Cuntz \cite{Cu2} introduced a general multiplicative bivariant Chern-Connes character on his bivariant $K$-theory for locally convex algebras with values in periodic cyclic cohomology.
A multiplicative bivariant Chern-Connes character 
$$
ch_{biv}:\,\,KK(-,-)\,\longrightarrow\,\,HC_{loc}(-,-)
\eqno(1.4)
$$
on Kasparov's $KK$-theory was finally constructed in \cite{Pu2}, \cite{Pu4}. 
It takes values in the bivariant local cyclic cohomology \cite{Pu4}, 
which was invented for that purpose (there can be no character on $KK$-theory with values in any of the other cyclic theories). 
The Chern-Connes character is uniquely determined by its naturality and multiplicativity. Again the target of the character can be computed by homological means. For finitely summable and Theta-summable Fredholm modules $ch_{biv}$ coincides with the various partial Chern characters in $K$-homology constructed before \cite{Pu5}.\\
\\
A basic question, which imposes itself, is: how much information is preserved by the Chern-Connes character ?\\
\\
While one cannot hope for any kind of surjectivity of a bivariant character, 
(the formal reasons are explained in \cite{M},7.10), the most optimistic claim would be that the Chern-Connes character is always rationally injective. In fact, this is the case for a quite large class of $C^*$-algebras, namely those which are $KK$-equivalent to commutative ones (\cite{Pu4}).\\
\\
The first examples of $C^*$-algebras which are not in this class were found by Skandalis \cite{Sk}. If $\Gamma$ is a word-hyperbolic group which has Kazhdan's Property $(T)$ (there exist many such groups), then the reduced group $C^*$-algebra $C^*_r\Gamma$ is not $KK$-equivalent to any nuclear (and in particular to any commutative) $C^*$-algebra. \\
\\
The main result of this paper states that these algebras already provide counterexamples to the rational injectivity of $ch_{biv}$.
\begin{theorem}
Let $\Gamma$ be a word-hyperbolic group with Kazhdan's Property $(T)$.
Then
$$
ch_{biv}:\,\,KK(C^*_r\Gamma,C^*_r\Gamma)\,\longrightarrow\,\,HC_{loc}(C^*_r\Gamma,C^*_r\Gamma)
\eqno(1.5)
$$
is not rationally injective. 
\end{theorem}
To show this, it is necessary to exhibit an element in the kernel of that map.
There is in fact a canonical candidate, which is derived from the "Gamma"-element $\gamma\in KK_\Gamma(\C,\C)$ of Kasparov \cite{Ka2}. The "Gamma"-element is idempotent under the Kasparov product and the idempotent element $j_r(\gamma)\in KK(C^*_r\Gamma,C^*_r\Gamma)$, obtained from it by descent \cite{Ka2}, acts as a natural projection onto the image of the Baum-Connes assembly map. The main result of Skandalis work, cited above, is  

\begin{theorem}(Skandalis \cite{Sk})
Let $\Gamma$ be a word-hyperbolic group with Kazhdan's Property $(T)$.
Let $\gamma\in KK_\Gamma(\C,\C)$ be Kasparov's "Gamma"-element and let\\  $j_r(\gamma)\in KK(C^*_r\Gamma,C^*_r\Gamma)$ be the element obtained from it by descent. Then
$$
j_r(\gamma)\,\neq\,1\,\in\,KK(C^*_r\Gamma,C^*_r\Gamma)\otimes\Q
\eqno(1.6)
$$
\end{theorem}

The proof of Skandalis' theorem uses only formal properties of $KK$. As such, it applies to various bifunctors, but not to $HC_{loc}$. So one may guess that this difference in the behavior of the two theories affects the properties of the Chern-Connes character. (This observation actually motivated the present work). In fact 

\begin{theorem}
Let $\Gamma$ be a word-hyperbolic group. Then
$$
ch_{biv}(j_r(\gamma))\,=\,1\,=\,ch_{biv}(1)\in\,HC_{loc}(C^*_r\Gamma,C^*_r\Gamma).
\eqno(1.7)
$$
\end{theorem}

The theorem appears as a rather straightforward consequence of Lafforgue's 
breakthrough \cite{La2} on the Baum-Connes conjecture with coefficients for word-hyperbolic groups. While there can be no essentially equivariant homotpy between the Gamma-element and the unit element among Hilbert modules with unitary $\Gamma$-action, Lafforgue exhibits (on 196 pages !) such a homotopy among Hilbert modules with $\Gamma$-action of weakly exponential growth. As local cyclic cohomology is flexible enough to deal with both kinds of homotopies, the conclusion follows.

\section{The Chern-Connes Character}

There exists a unique natural transformation of bifunctors 
$$
\begin{array}{cccc}
ch_{biv}: & KK(-,-) & \to & HC_{loc}(-,-)
\end{array}
\eqno(2.1)
$$
called the {\bf Chern-Connes Character}, on the category of separable, complex $C^*$-algebras from Kasparov's bivariant $K$-functor \cite{Ka1} to bivariant local cyclic cohomology, which is multiplicative and satisfies
$$
ch_{biv}(f_*^{KK})\,=\,f_*^{HC}
\eqno(2.2)
$$
for every homomorphism 
$f:A\to B$ of $C^*$-algebras, (\cite{Pu4}).\\
\\
An explicit construction of the Chern-Connes character is obtained as follows.
Let $A,B$ be separable $C^*$-algebras and let $\alpha\in KK(A,B)$ be a class in (even) bivariant $K$-theory. It can be represented by a bounded Fredholm module ${\cal E}=({\cal H}_B, \varphi,F)$ where ${\cal H}_B={\cal H}^+_B\oplus{\cal H}^-_B$ is a $\Z/2\Z$-graded Hilbert $B$-module,
$\varphi:A\to{\cal L(H}_B)$ is an even $*$-representation and $F\in{\cal L(H}_B)$ is an odd endomorphism satisfying $F^2=1$, 
$\varphi(A)(F^*-F)\subset {\cal K(H}_B)$, and $[F,\varphi(A)]\subset {\cal K(H}_B)$. One denotes by $\epsilon\in{\cal L(H}_B)$ the grading operator 
which equals $+1$ on the even and $-1$ on the odd part of ${\cal H}_B$.\\
\\
Following Cuntz \cite{Cu1}, one considers the universal extension of $C^*$-algebras 
$$
0\,\to\,qA\,\to\,QA\,\overset{id*id}{\longrightarrow}\,A\,\to\,0
\eqno(2.3)
$$
where $QA=A*A$ denotes the free product of two copies of $A$ in the category of $C^*$-algebras. To the given Fredholm module one may associate the homomorphisms 
$$
\phi_0:A\to {\cal L(H}_B),\,\,\,a\mapsto \frac{1+\epsilon}{2}\varphi(a)\frac{1+\epsilon}{2},
\eqno(2.4)
$$
$$
\phi_1:A\to {\cal L(H}_B),\,\,\,a\mapsto F\frac{1-\epsilon}{2}\varphi(a)\frac{1-\epsilon}{2}F.
\eqno(2.5)
$$
of $C^*$-algebras. The homomorphisms $q(\phi_0,\phi_1)$ and $Q(\phi_0,\phi_1)$  fit into a commutative diagram
$$
\begin{array}{ccccccccccc}
0 & \to & qA & \to & QA & \overset{id*id}{\longrightarrow} & A & \to & 0 \\
 & & \downarrow & & \downarrow & & \downarrow & & \\
0 & \to & {\cal K(H}_B) & \to & {\cal L(H}_B) & \to & {\cal Q(H}_B) & \to & 0 \\
\end{array}
\eqno(2.6)
$$
of morphisms of $C^*$-algebras. The map 
$$
\Phi_{\cal E}\,=\,q(\phi_0,\phi_1):\,qA\,\to\,{\cal K(H}_B)
\eqno(2.7)
$$ 
is called the {\bf characteristic homomorphism} associated to ${\cal E}$.\\
A well known argument of Joachim Cuntz \cite{Cu1} shows that the natural map 
$$
q(id,0):\,qA\to A
\eqno(2.8)
$$ 
is transformed into an isomorphism under any additive, split exact and matrix stable diffeotopy functor. In particular 
$$
q(id,0)_*\in HC_{loc}(qA,A)
\eqno(2.9)
$$ 
is invertible in the sense of \cite{Ka1}. Local cyclic cohomology is not only matrix stable as all cyclic homology theories, but satisfies the following much stronger stability condition. For any $*$-homomorphism of the form 
$$
i_p:B\to {\cal K(H}_B),\,b\mapsto pb,
\eqno(2.10)
$$
$p\in{\cal K(H}_B)$ a rank one projection, the element 
$$
i_{p*}\in HC_{loc}(B,{\cal K(H}_B))
\eqno(2.11)
$$ 
is invertible. (Note that $i_p$ depends on choices while $i_{p*}$ is unambiguously defined.) With this in mind we have the following formula for the bivariant Chern-Connes character:
$$
\begin{array}{ccccc}
ch_{biv}({\cal E}) & = &   i_{p*}^{-1}\circ\Phi_{{\cal E}*}\circ q(id,0)_*^{-1} & \in & HC^0_{loc}(A,B) \\
\end{array}
\eqno(2.12)
$$

\section{Fredholm representations of hyperbolic groups\\
(after V. Lafforgue)}

Let $\Gamma$ be a discrete group with finite symmetric set of generators $S$ and associated word length function $\ell_S$. Let $A$ be a 
separable Banach algebra equipped with an isometric linear $\Gamma$-action. Let finally $B$ be a $C^*$-algebra.
\begin{definition}
An (even) covariant Fredholm representation ${\cal E}\,=\,({\cal H},\varphi,\rho,F)$ of $(A,\Gamma)$ over $B$ of exponent $\lambda\geq 1$ 
is given by  
\begin{itemize}
\item a $\Z/2\Z$-graded Hilbert $B$-module ${\cal H}_B={\cal H}_B^+\oplus{\cal H}_B^-$,
\item an even representation $\varphi=\varphi_+\oplus\varphi_-:\,A\to{\cal L(H}_B)$ of $A$,
\item an even linear representation $\rho=\rho_+\oplus\rho_-:\,\Gamma\longrightarrow\,{\cal
L}({\cal H}_B)$ such that 
$$
 \parallel\rho(g)\parallel_{{\cal
L}({\cal H}_B)}\,=\,O(\lambda^{\ell_S(g)})
\eqno(3.1)
 $$ 
 and
 $$
 \begin{array}{cc}
 \rho(g)\varphi(a)\rho(g^{-1})\,=\,\varphi(g\cdot a), & \forall g\in\Gamma,\,\forall a\in A,
 \end{array}
 \eqno(3.2)
 $$
 \item an odd bounded operator $F\in{\cal
L}({\cal H}_B)$ such that 
\begin{itemize}
\item  $F^2=1$,
\item $F$ almost intertwines the representation $\varphi_+$ and $\varphi_-$, i.e.
$$
\begin{array}{cc}
[F,\varphi(a)]\,\in{\cal K(H}_B), & \forall a\in A. \\
\end{array}
\eqno(3.3)
$$

\item $F$ almost intertwines the representations $\rho_+$ and $\rho_-$, i.e.
$$
\begin{array}{cc}
\rho(g)F\rho(g)^{-1}\,-\,F \in{\cal K(H}_B), & \forall g\in \Gamma. \\
\end{array}
\eqno(3.4)
$$
\end{itemize}
\end{itemize}
\end{definition}

Suppose that a covariant Fredholm representation ${\cal E}\,=\,({\cal H}_B,\varphi, \rho,F)$ of $(A,\Gamma)$ over $B$ of exponent $\lambda\geq 1$ is given. 
Denote by $\epsilon\in{\cal L}({\cal H}_B)$ the grading operator on ${\cal H}_B$. Then
$\frac{1\pm\epsilon}{2}$ are the orthogonal projections on ${\cal H}_B^\pm$. They strictly commute with the actions 
of $A$ and $\Gamma$.

\begin{definition}
A differentiable (smooth) operator homotopy joining the covariant Fredholm representations ${\cal E}_0\,=\,({\cal H},\varphi,\rho,F_0)$ and ${\cal E}_1\,=\,({\cal H},\varphi,\rho,F_1)$ of $\Gamma$ over $B$ of exponent $\lambda\geq 1$
is a family ${\cal E}_t\,=\,({\cal H}_B, \varphi, \rho, F_t),$
$t\in[0,1],$ of Fredholm representations of exponent $\lambda$ such that $t\mapsto F_t$ is continuously differentiable (smooth) with respect to the operator norm on ${\cal L(H}_B)$.
\end{definition}

Let now $\Gamma$ be a word-hyperbolic group with a fixed finite symmetric set $S$ of generators and denote by $d_S$ the associated word length metric on $\Gamma$. The Cayley-graph ${\cal G}(\Gamma,S)$ is then a geodesic $\delta$-hyperbolic metric space for some $\delta>0$. Left translation yields an isometric action of $\Gamma$ on its Cayley-graph. For $N>>0$ let $\Delta=\Delta(N)$ be the Rips complex of $(\Gamma,S)$. 
Recall that $\Delta$ is the set of all nonempty (and necessarily finite) oriented subsets of $(\Gamma, d_S)$ of diameter at most $N$. (An orientation of a finite set $T$ of order $n$ is an equivalence class of total orderings of $T$, two orderings being equivalent if they are conjugate under the action of the alternating group $A_n$.) Let $x=\{e\}$ be the base simplex of $\Delta$. The usual simplicial face maps turn the Rips complex into a simplicial set $\Delta_\bullet$, given in degree n by the oriented subsets of $n$-elements and diameter at most $N$ in $(\Gamma, d_S)$. Note that $\Delta_n\,=\,\emptyset$ for $n>>0$. The action of $\Gamma$ on the Cayley graph gives rise to an isometric and simplicial action on $\Delta$. We let $\widetilde{\Delta}$ be the augmented Rips complex obtained by adding a unique simplex of degree 0, which represents the empty subset of $\Gamma$. 
Denote finally by $\C(\widetilde{\Delta})$ the graded $\C$-linear span of $\widetilde{\Delta}$, modulo the identification $e_S+e_T=0$ if $S,T\in\widetilde{\Delta}$ have the same underlying sets but opposite orientations, and let $(\C(\widetilde{\Delta}),\,\partial)$ be the chain complex associated to the simplicial set $\widetilde{\Delta}$. The natural linear action of $\Gamma$ on $\C(\widetilde{\Delta})$ will be denoted by $\rho$. It commutes with the simplicial differential $\partial$. A basic fact about hyperbolic groups is that $(\C(\widetilde{\Delta}),\,\partial)$ is contractible provided $N$ is large enough.\\
\\
In his monumental paper \cite{La2} Vincent Lafforgue has constructed the following homotopy of
 Fredholm representations, which he uses to establish the Baum-Connes conjecture with coefficients for all word-hyperbolic groups.

\begin{theorem}
Let $\Gamma$ be a word-hyperbolic group with finite symmetric set of generators $S$ and associated word-length metric $d_S$. 
Let $s>0$. Then there exist for $N>>0$ sufficiently large
\begin{itemize}
\item a contracting chain homotopy ("parametrix") $J_x$ of $(\C(\widetilde{\Delta}),\,\partial),$ i.e.
$$\partial J_x+J_x\partial=Id_{\C(\widetilde{\Delta})},$$
\item a metric $d^\flat:\Gamma\times \Gamma\to\R_+$ satisfying $d^\flat\,-\,d_S\,=\,O(1),$
\item a Hilbert space completion ${\cal H}_{x,s}$ of $\C(\widetilde{\Delta}),$
\end{itemize}
such that the following assertions hold for $T>0$: Lafforgue's family \cite{La2}
$$
\begin{array}{ccc}
{\cal E}_{s,t}\,=\,({\cal H}_{x,s},\,\varphi,\,\rho,\,F_t),& t\in[0,T], & F_t\,=\,e^{td^\flat(-,x)}(\partial+J_x\partial J_x)e^{-td^\flat(-,x)}
\end{array}
\eqno(3.5)
$$
defines a differentiable (and in fact smooth) operator homotopy of Fredholm representations of $(\C,\Gamma)$ of exponent $\lambda=e^s$. 
\end{theorem}
\begin{remark}
Note that the family (3.5) differs slightly from the one considered by Lafforgue, because we use the augmented instead of the ordinary Rips complex. The operator $\varphi(1)$ is given by the orthogonal projection 
onto the span of $\C(\Delta)$. With this modification the operator $F_t$ of (3.5) satisfies $F_t^2=1$ for all $t\geq 0$.  
\end{remark}
\begin{proof}
The theorem is implicitly contained in Lafforgue's monumental paper \cite{La2}. The main result of that paper states that the family (3.5) defines a Fredholm representation of $(\C,\Gamma)$ over $\C[0,T]$ of exponent 
$\lambda=e^s$. So the only issue is to verify that the family of operators $F_t$ varies continuously (smoothly) with repsect to the operator norm on ${\cal H}_{x,s}$. The notations of Lafforgue's paper are from now on understood. 
Citations refer exclusively to \cite{La2}. The constants that come up depend on the various parameters of Lafforgue's construction, i.e.
$C=C(\delta,K,N,Q,P,M,s,B,T)$. We proceed as in the proof of his lemma 4.47.

It has to be shown that for 
$$
F\,=\,\partial+J_x\partial J_x 
\eqno(3.6)
$$ 
the operator families
$$
\begin{array}{ccc}
F_t & = & \left( e^{td_x^{\flat}}\,F\,e^{-td_x^{\flat}}\right),
\end{array}
\eqno(3.7)
$$
and
$$
\begin{array}{ccc}
\frac{d}{dt}F_t & = & \left( e^{td_x^{\flat}}\,[d_x^{\flat},F]\,e^{-td_x^{\flat}}\right)
\end{array}
\eqno(3.8)
$$
are continuous in $t$ with respect to the operator norm on ${\cal L(H}_{x,s})$.
Here $d_x^{\flat}$ denotes the diagonal operator on $\C(\Delta)$ which multiplies the basis vector $e_S$ corresponding to a Rips simplex $S\in\Delta$
with the averaged distance of the simplex $S$ from the origin $x$. Note that $d^{\flat}$ is the "continuous metric" associated to the original discrete word-length metric $d$ on $\Gamma$
\cite{La2}, section 3.5. The difference $d_x^{\flat}-d_x$ of the unbounded multiplication operators $d_x^{\flat}$ and $d_x$ is in fact a bounded operator on ${\cal H}_{x,s}$
\cite{La2}, Lemma 4.49. Thus
$$
F_t\,=\,e^{t(d_x^{\flat}-d_x)}\left(e^{td_x}Fe^{-td_x}\right)e^{-t(d_x^{\flat}-d_x)}
\eqno(3.9)
$$
and
$$
\frac{d}{dt}F_t\,=\,e^{t(d_x^{\flat}-d_x)}\,[d_x^{\flat}-d_x,\,e^{td_x}Fe^{-td_x}]\,e^{-t(d_x^{\flat}-d_x)}
$$
$$
+\,e^{t(d_x^{\flat}-d_x)}\left(e^{td_x}[d_x,F]e^{-td_x}\right)e^{-t(d_x^{\flat}-d_x)}
\eqno(3.10)
$$
because the operators $d_x^{\flat}$ and $d_x$ commute.
Now
$F=\partial+J_x\partial J_x$
where $\partial$ denotes the ($\Gamma$-equivariant) simplicial differential of the cellular chain complex $C_*(\widetilde{\Delta})$ associated to $\widetilde{\Delta}$, and 
$J_x$ is a well chosen (non equivariant) contracting chain homotopy (called a "parametrix" in \cite{La2}) of $C_*(\widetilde{\Delta})$.
According to \cite{La2}, section 3.4, 
$$
\begin{array}{ccccc}
J_x & = & \underset{q=1}{\overset{Q}{\sum}}\,\widetilde{H}_{x,q} & + & \underset{r=1}{\overset{\infty}{\sum}}\,u_{x,r}K_{x,Q}
\end{array}
\eqno(3.11)
$$
for $Q\in\N$ large, but fixed. 
The theorem is therefore implied by the norm continuity of the families 
$$
e^{td_x}\partial e^{-td_x},\,e^{td_x}\widetilde{H}_{x,q}e^{-td_x},\,e^{td_x}u_{x,r}K_{x,Q} e^{-td_x},\, 1\leq q\leq Q,\,r\geq 1,
\eqno(3.12)
$$
and
$$
e^{td_x}[d_x,\partial]\,e^{-td_x},\,e^{td_x}[d_x,\widetilde{H}_{x,q}]\,e^{-td_x},\,e^{td_x}[d_x,u_{x,r}K_{x,Q}]\,e^{-td_x},\, 1\leq q\leq Q,\,r\geq 1,
\eqno(3.13)
$$
of operators.
 
Let $P_R$ denote the orthogonal projector of ${\cal H}_{x,s}$ onto the subspace spanned by Rips simplices of distance at most $R$ from the origin $x$. Put ${\cal P}=P_P$, \cite{La2}, p.91.\\
\\
{\bf Continuity of $t\mapsto e^{td_x}\partial e^{-td_x}$:}\\
The operator $\partial$ on $\C(\Delta^p),\,p>1,$  is of propagation at most $N$ so that $P_R\partial\,=\,P_R\partial P_{R+N}$. (The case $p=1$ is left to the reader as it will not be used in the sequel.) The restriction of the 
diagonal operator $d_x$ to the image of $P_{R+N}$ is bounded so that $t\mapsto e^{td_x}{\cal P}\partial e^{-td_x}$ and 
$t\mapsto e^{td_x}{\cal P}[d_x,\partial] e^{-td_x}$ are norm continuous. 
The calculations in the proof of lemma 4.21 yield, after replacing in formula (55) the expression $\xi_{\widetilde Z}(f)$ by 
$\left(e^{t(\rho_x^1(\widetilde{Z})-\rho_x^0(\widetilde{Z}))}\,-\,e^{t'(\rho_x^1(\widetilde{Z})-\rho_x^0(\widetilde{Z}))}\right)\xi_{\widetilde Z}(f)$
 and noting that $\vert\rho_x^1(\widetilde{Z})-\rho_x^0(\widetilde{Z})\vert\leq N$ (see page 17),
the estimate
$$
\parallel e^{td_x}(1-{\cal P})\partial e^{-td_x}\,-\,e^{t'd_x}(1-{\cal P})\partial e^{-t'd_x}\parallel_{{\cal L}({\cal H}_{x,s})}
$$
$$
\leq\,C_0\,\underset{\vert \sigma\vert\leq N}{\Sup}\,\vert e^{t\sigma}-e^{t'\sigma}\vert\,
\leq\,C_1\vert t'-t\vert.
\eqno(3.14)
$$
A similar argument shows
$$
\parallel e^{td_x}(1-{\cal P})[d_x,\partial] e^{-td_x}\,-\,e^{t'd_x}(1-{\cal P})[d_x,\partial] e^{-t'd_x}\parallel_{{\cal L}({\cal H}_{x,s})} 
$$
$$
\leq\,C_1N\vert t'-t\vert.
\eqno(3.15)
$$
{\bf Continuity of $t\mapsto e^{td_x}\widetilde{H}_{x,q} e^{-td_x}$:}\\
The operator $\widetilde{H}_{x,q} $ is of finite propagation. The same argument as before shows therefore the continuity of 
$t\mapsto e^{td_x}{\cal P}\widetilde{H}_{x,q} e^{-td_x}$ and $t\mapsto e^{td_x}{\cal P}[d_x,\widetilde{H}_{x,q}] e^{-td_x}$.
The calculations in the proof of lemma 4.36 yield, after replacing in formula (95) the expression $\xi_{\widetilde Z}(f)$ by 
$\left(e^{t(\rho_x^1(\widetilde{Z})-\rho_x^0(\widetilde{Z}))}\,-\,e^{t'(\rho_x^1(\widetilde{Z})-\rho_x^0(\widetilde{Z}))}\right)\xi_{\widetilde Z}(f)$
 and noting that\\
$\vert\rho_x^1(\widetilde{Z})-\rho_x^0(\widetilde{Z})\vert\leq (q+2)N$ (see (84) and top of page 122), the estimate
$$
\parallel e^{td_x}(1-{\cal P})\widetilde{H}_{x,q} e^{-td_x}\,-\,e^{t'd_x}(1-{\cal P})\widetilde{H}_{x,q} e^{-t'd_x}\parallel_{{\cal L}({\cal H}_{x,s})}
$$
$$
\leq\,C_2\,\underset{\vert \sigma\vert\leq (q+2)N}{\Sup}\,\vert e^{t\sigma}-e^{t'\sigma}\vert\,\,\leq\,C_3\,\vert t'-t\vert.
\eqno(3.16)
$$
A similar argument shows
$$
\parallel e^{td_x}(1-{\cal P})[d_x,\widetilde{H}_{x,q}] e^{-td_x}\,-\,e^{t'd_x}(1-{\cal P})[d_x,\widetilde{H}_{x,q}] e^{-t'd_x}\parallel_{{\cal L}({\cal H}_{x,s})}
$$
$$
\leq\, C_3(q+2)N\vert t'-t\vert.
\eqno(3.17)
$$
\\
{\bf Continuity of $e^{td_x}u_{x,r}K_{x,Q} e^{-td_x},\,r\geq 1$:}\\
One deduces from the calculations in the proof of 4.32, after replacing in formula (80) the expression $\xi_{Z}(f)$ by
 $\left(e^{t(\rho_x(U)-\rho_x^0(Z))}\,-\,e^{t'(\rho_x(U)-\rho_x^0(Z))}\right)\xi_{Z}(f)$, and noting 
 that $\vert r+\rho_x(U)-\rho_x^0(Z)\vert\leq QF+N$ (see top of pages 107 and 122),
the estimate
$$
\parallel e^{td_x}{\cal P}u_{x,r}K_{x,Q} e^{-td_x}\,-\,e^{t'd_x}{\cal P} u_{x,r}K_{x,Q} e^{-t'd_x}\parallel_{{\cal L}({\cal H}_{x,s})}
$$
$$
\leq\,C_4\,e^{-\frac{s}{2}r}\,\underset{\vert r+\sigma\vert\leq QF+N}{\Sup}\,\vert e^{t\sigma}-e^{t'\sigma}\vert\,\,\leq\, C_5\,e^{-\frac{s}{2}r}\,(r+C_6)\vert t'-t\vert.
\eqno(3.18)
$$
The same reasoning based on the calculations in the proof of 4.40 and the substitution 
$\xi_{\widetilde Z}(f)\mapsto \left(e^{t(\rho_x^1(\widetilde{Z})-\rho_x^0(\widetilde{Z}))}\,-\,e^{t'(\rho_x^1(\widetilde{Z})-\rho_x^0(\widetilde{Z}))}\right)\xi_{\widetilde Z}(f)$ 
in formula (112) yields, after noting that $\vert r+\rho_x^1(\widetilde{Z})-\rho_x^0(\widetilde{Z})\vert \leq\,QF+N$ (see (98) and top of page 122), 
$$
\parallel e^{td_x}(1-{\cal P})u_{x,r}K_{x,Q} e^{-td_x}\,-\,e^{t'd_x}(1-{\cal P})u_{x,r}K_{x,Q} e^{-t'd_x}\parallel_{{\cal L}({\cal H}_{x,s})}
$$
$$
 \leq\, C_7\,e^{-\frac{s}{2}r}\,(r+C_8)\vert t'-t\vert.
\eqno(3.19)
$$
Similar arguments show finally
$$
\parallel e^{td_x}[d_x,{\cal P}u_{x,r}K_{x,Q}]e^{-td_x}\,-\,e^{t'd_x}[d_x,{\cal P}u_{x,r}K_{x,Q}] e^{-t'd_x}\parallel_{{\cal L}({\cal H}_{x,s})}
$$
$$
 \leq\, C_5\,e^{-\frac{s}{2}r}\,(r+C_6)(r+QF+N)\vert t'-t\vert.
\eqno(3.20)
$$
and
$$
\parallel e^{td_x}[d_x,(1-{\cal P})u_{x,r}K_{x,Q}]e^{-td_x}\,-\,e^{t'd_x}[d_x,(1-{\cal P})u_{x,r}K_{x,Q}] e^{-t'd_x}\parallel_{{\cal L}({\cal H}_{x,s})}
$$
$$
 \leq\, C_7\,e^{-\frac{s}{2}r}\,(r+C_8)(r+QF+N)\vert t'-t\vert.
\eqno(3.21)
$$

The assertion of the theorem is then a direct consequence of the estimates (3.14) to (3.21) above. The proof implies that Lafforgue's operator homotopy is actually smooth, i.e. $t\mapsto F_t\,\in\,{\cal C}^\infty([0,T],{\cal L}({\cal H}_{x,s}))$.
\end{proof}

\begin{lemma}(Lafforgue,\cite{La2})
Fix $s>0$ and let ${\cal E}_{s,t}\,=\,({\cal H}_{x,s},\,\varphi,\,\rho,\,F_t), t\geq 0,$ be Lafforgue's Fredholm representation (see (3.5)). Then for $T>>0$ sufficiently large
$$
{\cal E}_T\,=\,(\ell^2(\widetilde{\Delta}),\,\varphi,\,\rho,\,F_T)
$$
is a Fredholm representation of $(\C,\Gamma)$ of exponent 1
which is piecewise $C^1$-operator homotopic to the equivariant Kasparov module  $\gamma_{KS}$ of \cite{KaSk} representing Kasparov's $\gamma$-element
$$
\gamma \in KK_\Gamma(\C,\C).
$$ 
\end{lemma}

\begin{proof}
This is proved in \cite{La2}, section 5.
\end{proof}
\begin{lemma}(Lafforgue,\cite{La1})
Let ${\cal E}_{s,0}=({\cal H}_{x,s},\,\varphi,\,\rho,\,F),\,F=\partial+J_x\partial J_x,$ be the Fredholm representation introduced in (3.5). 
Then 
$$
{\cal E}_{s,0,t}\,=\,({\cal H}_{x,s},\,\varphi,\,\rho_t,\,F),\,t\in[0,1],
$$
$$
\rho_t\,=\,(1-t)\rho\,+\,t\left((J_x\partial)\rho(J_x\partial)+(\partial J_x)\rho(\partial J_x)\right),
$$
defines a Fredholm representation of $(\C,\Gamma)$ over $C^1([0,1])$ of exponent $\lambda$ 
such that ${\cal E}_{s,0,0}={\cal E}_{s,0}$ and such that the operator $F$ strictly commutes with $\rho_1$ and $\varphi$.
\end{lemma}
\begin{proof}
This follows from the fact that $\partial J_x$ and $J_x\partial=1-\partial J_x$ are idempotent operators which commute with $\varphi(\C)$ and $\rho_1(\Gamma)$ (see Lemma 1.4.2 of \cite{La1}).
\end{proof}

\section{Unconditional Banach algebras and\\ characteristic homomorphisms}

\begin{definition} \cite{La1}
Let $(\Gamma,S)$ be a finitely generated discrete group with finite symmetric set of generators $S$. A Banach algebra $A\Gamma\subset C_r^*\Gamma$ containing the group ring $\C\Gamma$ as dense subalgebra is called {\bf unconditional} \cite{La1} if 
$$
\begin{array}{ccc}
\vert a_g\vert\,\leq\,\vert b_g\vert,\,\forall g\in\Gamma & \Rightarrow & \parallel \underset{g}{\sum} a_gu_g\parallel\,\leq\,\parallel \underset{g}{\sum} b_gu_g\parallel.
\end{array}
\eqno(4.1)
$$
\end{definition}

\begin{definition}
Let $\pi:\Gamma\to Isom(A)$  be an isometric linear action of $\Gamma$ on the Banach algebra $A$.
Fix $\lambda\geq 1$. We denote ${\cal A}_\lambda(\Gamma,A)$ the Banach algebra obtained by completion 
of the algebraic crossed product $A\rtimes\Gamma$ with respect to the submultiplicative seminorm 
$$
\parallel\underset{g}{\sum}\,a_gu_g\parallel_{{\cal A}_\lambda(\Gamma,B)}\,=\,\parallel\underset{g}{\sum}\lambda^{\ell_S(g)}\parallel a_g\parallel_A\,u_g\parallel_{{\cal A}\Gamma}.
\eqno(4.2)
 $$
\end{definition}
The Banach algebra corresponding to $\lambda=1$ will simply be denoted ${\cal A}(\Gamma,A)$.\\
\\
Recall that the Cuntz-algebra $Q_RA$ of a Banach algebra $A$ \cite{Pu5} is the completion of the algebraic free product $A*A$ with respect to the largest Banach algebra norm 
such the canonical inclusions $i_{0,1}:A\to QA$ are bounded of norm $\leq R$. The Cuntz algebra fits into a doubly split extension of Banach algebras
$$
0\,\to\,q_RA\,\to\,Q_RA\,\overset{id*id}{\longrightarrow}\,A\,\to\,0
\eqno(4.3)
$$
If $f,g:A\to B$ are homomorphisms of Banach algebras of norm $\leq R$, which coincide modulo an ideal $I\subset B$, then the restriction of 
$$
Q_R(f,g):Q_RA\to B
\eqno(4.4)
$$ 
to the Cuntz ideal yields a homomorphism
$$
q_R(f,g):q_RA\to I.
\eqno(4.5)
$$
If $f,g:A\to B$ are homomorphisms which are orthogonal to each other in the sense that 
$$
f(A)g(A)=g(A)f(A)=0,
\eqno(4.6)
$$
then $f+g$ is a homomorphism of Banach algebras as well and 
$$
q_R(f+g,g)\,=\,q_R(f,0)
\eqno(4.7)
$$
for $R>>0$.

\begin{deflem}
Let ${\cal E}=({\cal H}_B,\,\varphi,\,\rho,\,F)$ be a Fredholm representation of $(A,\Gamma)$ over $B$ of exponent $\lambda$. Let $\epsilon\in{\cal L(H}_B)$ be the associated grading operator. 

\begin{itemize}
\item[a)]
Define algebra homomorphisms
$$
\psi_0:A\rtimes\Gamma\to{\cal L}({\cal H}_B)\otimes\C\Gamma,\,a_gu_g\mapsto\frac{1+\epsilon}{2}\varphi(a_g)\rho(g)\frac{1+\epsilon}{2}\otimes u_g,
\eqno(4.8)
$$
and
$$
\psi_1:A\rtimes\Gamma\to{\cal L}({\cal H}_B)\otimes\C\Gamma,\,a_gu_g\mapsto F\frac{1-\epsilon}{2}\varphi(a_g)\rho(g)\frac{1-\epsilon}{2}F\otimes u_g.
\eqno(4.9)
$$
These extend to bounded homomorphisms of Banach algebras
$$
\psi_{0},\psi_1\,:{\cal A}_\lambda(\Gamma, A)\to{\cal A}(\Gamma,{\cal L}({\cal H}_B))
\eqno(4.10)
$$
(Here we view ${\cal L}({\cal H}_B)$ as Banach algebra equipped with the trivial $\Gamma$ action.)
\item[b)]
The restriction of $\psi_0*\psi_1:Q_R({\cal A}_\lambda(\Gamma,A))\to{\cal A}(\Gamma,{\cal L}({\cal H}_B))$ to the Cuntz ideal $q_R({\cal A}_\lambda(\Gamma,A))$ gives rise to a bounded homomorphism
$$
\Psi_{{\cal E}}=q_R(\psi_0,\psi_1):q_R({\cal A}_\lambda(\Gamma,A))\to{\cal A}(\Gamma,{\cal K}({\cal H}_B)),\,\,\,R>>0,
\eqno(4.11)
$$
of Banach algebras, called the {\bf characteristic homomorphism} associated to the Fredholm representation $\cal E$.
\end{itemize}
\end{deflem}

This is clear from the definitions.

\begin{lemma}
Let ${\cal E}(t)=({\cal H}_B,\,\varphi_t,\,\rho_t,\,F_t),\,t\in [0,1],$ be a family of Fredholm representations 
of $(\Gamma,A)$ over $B$ of exponent $\lambda$. Suppose that $\widetilde{F}: t\mapsto F_t,$
$\widetilde{\rho}:(t,g)\mapsto \lambda^{-\ell_S(g)}\rho_t(g),$ and $\widetilde{\varphi}:(t,a)\mapsto \varphi_t(a)$ are continuously differentiable in the sense that
$\widetilde{F}\in C^1([0,1],{\cal L}({\cal H}_B)),$
$\widetilde{\rho}\in C^1([0,1],\ell^\infty(\Gamma,{\cal L}({\cal H}_B))),$
and $\widetilde{\varphi}\in C^1(I,{\cal L}(A,{\cal L}({\cal H}_B))),$
respectively. Then 
$$
\Psi_{{\cal E}(t)}:q_R({\cal A}_\lambda(\Gamma,A))\to C^1( [0,1], {\cal A}(\Gamma,{\cal K}({\cal H}_B))),\,\,\,R>>0,
\eqno(4.12)
$$
is bounded and yields a differentiable homotopy between the characteristic homomorphisms associated to ${\cal E}_0$ and ${\cal E}_1$. 
\end{lemma}

\begin{proof}
If ${\cal E}_t=({\cal H}_B,\,\varphi_t,\,\rho_t,\,F_t),\,t\in [0,1],$ is a differentiable family as described in the assumptions of the lemma, then formulas (4.8), (4.9) define maps
$$
\psi_0,\psi_1:A\rtimes\Gamma\to C^1([0,1],{\cal L}({\cal H}_B))\otimes\C\Gamma
\eqno(4.13)
$$ 
which extend to bounded homomorphisms 
$$
\psi_0,\psi_1:{\cal A}_\lambda(\Gamma, A)\to{\cal A}(\Gamma,C^1([0,1],{\cal L}({\cal H}_B)))
\eqno(4.14)
$$ 
and 
$$
q_R(\psi_0,\psi_1):q_R({\cal A}_\lambda(\Gamma,A))\to{\cal A}(\Gamma,C^1([0,1],{\cal K}({\cal H}_B))),\,\,\,R>>0.
\eqno(4.15)
$$ 
Composition with the tautological bounded homomorphism 
$$
{\cal A}(\Gamma,C^1([0,1],{\cal K}({\cal H}_B)))\to C^1([0,1],{\cal A}(\Gamma,{\cal K}({\cal H}_B)))
\eqno(4.16)
$$ 
yields the claim.
\end{proof}

The previous lemma applies in particular to the operator homotopies of 3.3 and the more general homotopy of 3.6.

\begin{theorem}
Let $\Gamma$ be a word-hyperbolic group and let ${\cal A}\Gamma$ be an unconditional Banach algebra over $\Gamma$.
Let $\gamma_{KS}$ be the Fredholm module(Fredholm representation of exponent 1) of \cite{KaSk} representing the $\gamma$-element $\gamma\in KK_\Gamma(\C,\C)$ and let 
$$
\Psi_{\gamma_{KS}}:\,q({\cal A}\Gamma)\,\to\,{\cal A}(\Gamma,{\cal K}(\ell^2(\widetilde{\Delta}))
\eqno(4.17)
$$ 
be the associated characteristic homomorphism (see (4.11)). Then for every $\lambda>1$ there exists $R>>0$ such that the composition 
$$
q_R({\cal A}_\lambda\Gamma)\,\to\,q_R({\cal A}\Gamma)\,\overset{\Psi_{\gamma_{KS}}}{\longrightarrow}\,{\cal A}(\Gamma,{\cal K}(\ell^2(\widetilde{\Delta})))
\eqno(4.18)
$$
is piecewise $C^1$-homotopic to the composition of homomorphisms 
$$
q_R({\cal A}_\lambda\Gamma)\overset{q(id,0)}{\longrightarrow}{\cal A}_\lambda\Gamma\overset{i_p}{\longrightarrow}
{\cal A}(\Gamma,{\cal K}(\ell^2(\widetilde{\Delta}))),\,\,\,i_p(u_g)\,=\,p\otimes u_g,
\eqno(4.19)
$$
where $p$ is a rank one projection. 
\end{theorem}

\begin{proof}
This follows from the results of section 2. Let $s=\log(\lambda)>0$ and consider Lafforgue's family of Fredholm representations of exponent $\lambda$. In fact, 3.5 and 4.4 imply that the homomorphism (4.17) is piecewise $C^1$-homotopic to $\Psi_{{\cal E}_T}$ for $T>>0$. This Fredholm representation is defined by the same representations and operators on the pre-Hilbert space $\C(\widetilde{\Delta})$ as Lafforgue's Fredholm representation ${\cal E}_{s,T}$. These two representations take values in different Hilbert space completions of $\C(\widetilde{\Delta})$ though. In order to compare the associated characteristic homomorphisms it is necessary to identify the various Hilbert spaces coming up with our prefered model Hilbert space. This implies that the various characteristic homomorphisms take values in the same algebra. The new characteristic homomorphisms will however be well defined only up to inner automorphisms. As conjugate homomorphisms into ${\cal K}$ are smoothly homotopic to each other this does not affect our considerations. Thus (4.18) is homotopic to 
$\Psi_{{\cal E}_{s,T}}$ and further to $\Psi_{{\cal E}_{s,0}}$ and $\Psi_{{\cal E}_{s,0,1}}$ by 3.3 and 3.6. We calculate now the latter homomorphism. One has
$$
{\cal E}_{s,0,1}\,=\,({\cal H}_{x,s},\,\varphi,\,\rho',\,F),\,\,\,\rho'=
\left((J_x\partial)\rho(J_x\partial)+(\partial J_x)\rho(\partial J_x)\right),\,\,\,F=(\partial+J_x\partial J_x).
\eqno(4.20)
$$
Let $p_x:\C(\widetilde{\Delta})\to\C(\widetilde{\Delta})$ be the idempotent operator of rank one which vanishes on $\C(\Delta_{k})$ for $k\neq 1$ and is given on $\C(\Delta_{1})=\C(\Gamma)$ by the projection with image the subspace spanned by the base simplex $e_x=u_e$ and kernel the augmentation ideal of $\C(\Gamma)$. Put $p_x\otimes id:\C\Gamma\to{\cal K(H)}\otimes\C\Gamma,\,u_g\mapsto p_x\otimes u_g.$
A straightforward calculation (making use of the fact that $\varphi(1)$ and $F$ strictly intertwine the representation $\rho'$) shows that the homomorphisms 
$\psi_{0,1}:\C\Gamma\to{\cal L(H)}\otimes\C\Gamma$ attached by (4.8) and (4.9) to the Fredholm representation ${\cal E}_{s,0,1}$ satisfy 
$$
\psi_0=\psi_1+p_x\otimes id 
\eqno(4.21)
$$
and
$$
\psi_1(\C\Gamma)(p_x\otimes id)(\C\Gamma)=
(p_x\otimes id)(\C\Gamma)\psi_1(\C\Gamma)=0.
\eqno(4.22)
$$ 
So we may conclude from (4.7) that 
$$
q(\psi_0,\psi_1)=q(\psi_1+p_x\otimes id,\psi_1)=q(p_x\otimes id,0),
\eqno(4.23)
$$
i.e. the homomorphism $q(\psi_0,\psi_1)$  
equals the composition
$$
q(\C\Gamma)\overset{q(id,0)}{\longrightarrow}
\C\Gamma\overset{p_x\otimes id}{\longrightarrow}{\cal K(H)}\otimes\C\Gamma.
\eqno(4.24)
$$
Passing to completions we find that the characteristic homomorphism attached to ${\cal E}_{s,0,1}$ factors for $R>>0$ as
$$
\Psi_{{\cal E}_{s,0,1}}:\,q_R({\cal A}_\lambda\Gamma)\,\overset{q_R(id,0)}{\longrightarrow}\,{\cal A}_\lambda(\C,\Gamma)\,
\overset{{\cal A}(p_x\otimes id)}{\longrightarrow}\,{\cal A}({\cal K(H)},\Gamma).
\eqno(4.25)
$$
Finally the idempotent rank one operator $p_x$ is smoothly homotopic 
to a rank one projection, and the theorem follows.

\end{proof}

\section{Local cyclic cohomology}
We recall some basic facts about local cyclic cohomology.

Cuntz and Quillen \cite{CQ2} define the $X$-complex of a complex Banach algebra $R$ as the $\Z/2\Z$-graded chain complex of Banach spaces 
$$
\begin{array}{cccc}
X_*(R): & X_0(R)\,=\,R & 
\leftrightarrow
& X_1(R)\,=\,\Omega^1R/\overline{[\Omega^1R,R]}
\end{array}
\eqno(5.1)
$$
with differentials
$$
\begin{array}{cccccc}
\overline{d}:& R & \overset{d}{\to} & \Omega^1R & \to & \Omega^1R/\overline{[\Omega^1R,R]}
\end{array}
\eqno(5.2)
$$
and
$$
\begin{array}{ccccccc}
b: & \Omega^1R/\overline{[\Omega^1R,R]} & \to & R, & \overline{a_0da_1} &
\mapsto & [a_0,a_1].
\end{array}
\eqno(5.3)
$$

Let $A$ be a Banach algebra with open unit ball $U$. For any compact subset $K\subset U$ denote by $A_K$ the completion of the complex subalgebra $\C[K]$ of $A$ generated by $K$ with respect to the largest submultiplicative seminorm $\parallel-\parallel_K$ satisfying $\parallel K\parallel_K\leq 1$. For $K\subset K'\subset U$ the identity on $\C[K]$ gives rise to a diagram $A_K\to A_{K'}\to A$ so that we obtain an ind-algebra
$$
\begin{array}{ccc}
{\cal A} & = & "\underset{\underset{\cal S}{\longrightarrow}}{\lim}"\,A_K
\end{array}
\eqno(5.4)
$$
parametrised by the family ${\cal S}\,=\,{\cal S}(A)\,=\,\{K\subset U\subset A,\,K\,\text{compact}\}$ of compact subsets of the open unit ball of $A$, ordered by inclusion (\cite{Pu4},1.3). It comes equipped with a canonical morphism
$$
{\cal A} \,\to \,A.
\eqno(5.5)
$$
Let $B$ be a Banach algebra and let $R>1$. We denote by $T_RB$ the Banach algebra obtained by completion of the tensor algebra 
$TB\,=\,\underset{n\geq 1}{\bigoplus}\,B^{\otimes n}$ with respect to the largest submultiplicative seminorm such that the canonical linear inclusion
$$
\varrho:\,B\to T_RB
\eqno(5.6)
$$ 
is of norm at most 2 and such that the linear map
$$
\begin{array}{ccccccc}
\omega:& B\otimes_\pi B & \to & T_RB, & (b_0\otimes b_1) & \mapsto & 
\varrho(b_0b_1)\,-\,\varrho(b_0)\varrho(b_1)
\end{array}
\eqno(5.7)
$$
is of norm at most $R^{-1}$. For $1<R<R'$ the identity on $TB$ gives rise to a bounded homomorphism $T_RB\to T_{R'}B$ so that one can form the ind-algebra
(\cite{Pu4},1.6)
$$
\begin{array}{ccc}
{\cal T}B & = & "\underset{R\to\infty}{\lim}"\,T_RB
\end{array}
\eqno(5.8)
$$
For a Banach algebra $A$ one calls 
$$
\begin{array}{ccc}
{\cal TA} & = & "\underset{R\to\infty}{\lim}"
"\underset{\underset{\cal S}{\longrightarrow}}{\lim}"\,T_R(A_K)
\end{array}
\eqno(5.9)
$$
the universal infinitesimal deformation of $A$ (\cite{Pu4},1.23).
\\
The formal inductive limit 
$$
\begin{array}{ccc}
X_*({\cal TA}) & = & "\underset{R\to\infty}{\lim}"
"\underset{\underset{\cal S}{\longrightarrow}}{\lim}"\,X_*(T_R(A_K))
\end{array}
\eqno(5.10)
$$
of chain complexes is called the {\bf analytic cyclic bicomplex} of $A$.\\
\\
Consider the ind-category $ind\,{\cal C}$ of formal inductive limits of $\Z/2\Z$-graded chain complexes of Banach spaces over $\C$. Its chain homotopy category $Ho(ind\,{\cal C})$ is triangulated in a natural way. 
We call an ind-complex weakly contractible if every chain map from a constant ind-complex to it is nullhomotopic. The weakly contractible ind-complexes form a nullsystem and the triangulated category obtained from $Ho(ind\,{\cal C})$ by inverting all chain maps with weakly contractible mapping cone is called the {\bf derived ind-category} $ind\,{\cal D}$ (\cite{Pu4},5.5). \\
\\
The {\bf bivariant local cyclic cohomology} of a pair $(A,B)$ of complex Banach algebras is defined (see \cite{CQ2},\cite{Pu4}) as 
$$
HC_{loc}^*(A,B)\,=\,Mor_{ind\,{\cal D}}(X_*({\cal TA}),X_*({\cal TB})).
\eqno(5.11)
$$
Composition of morphisms gives rise to an associative bilinear product 
$$
HC_{loc}^*(A,B)\,\otimes\,HC_{loc}^*(B,C)\,\longrightarrow\,HC_{loc}^*(A,C).
\eqno(5.12)
$$
One says that a homomorphism of ind-Banach algebras is a {\bf local 
$HC$-equivalence} if the corresponding morphism of analytic cyclic (ind-)complexes is an isomorphism in the derived ind-category. 

In order to verify this the following criterion is useful.

\begin{prop}(\cite{Pu4},2.9, 5.5)
Let 
$$
\Psi:\,"\underset{\underset{I}{\longrightarrow}}{\lim}"C^{(i)}_*\,
\longrightarrow\,"\underset{\underset{J}{\longrightarrow}}{\lim}"E^{(j)}_*
$$ 
be a chain map of degree $n\in\{0,1\}$ of ($\Z/2\Z$-graded) ind-complexes.
Then $\psi$ defines an isomorphism in the derived ind-category $ind\,{\cal D}$ if and only if the following condition is satisfied:

For given $i\in I,\,j\in J,$ and $\psi^{ij}:C^{(i)}_*\to E^{(j)}_*$ representing $\Psi$ there exist\\ $i'\geq i\in I,\,j'\geq j\in J,$ and a chain map 
$\chi^{ii'jj'}:E^{(j)}[n]_*\,\to\,C^{(i')}_*$ of degree $-n$ such that the diagram 
$$
\begin{array}{ccc}
C^{(i')}_* & \overset{\psi^{i'j'}}{\longrightarrow} & E^{(j')}[n]_* \\
 & & \\
\uparrow & \chi^{ii'jj'}\nwarrow & \uparrow \\
& & \\

C^{(i)}_* & \overset{\psi^{ij}}{\longrightarrow} & E^{(j)}[n]_* \\
\end{array}
\eqno(5.13)
$$
commutes up to chain homotopy. (The vertical maps are the structure maps of the ind-complexes. No compatibility of the various diagrams among each other is required.) 
\end{prop}

\begin{proof}
For the facts about triangulated categories, which are used in this proof, we refer to \cite{KaSh}, 1.4.-1.6. It is shown in \cite{Pu4}, 5.5 and 2.9, that a chain map of ind-complexes, which satisfies criterion (5.13) defines an isomorphism in the derived ind-category. Let now $f:X\to Y$ be a chain map of ind-complexes, which defines an isomorphism in the derived ind-category $ind\,{\cal D}$. This implies that its cone $Cone(f)$ equals 0 in $ind\,{\cal D}$. Therefore, there exists a chain map of ind-complexes $\varphi:Z\to Cone(f)$ with weakly contractible mapping cone $Cone(\varphi)$, such that $\varphi=id_{Cone(f)}\circ\varphi=0_{Cone(f)}\circ\varphi=0$. The cone of 
$Z\overset{0}{\to}Cone(f)$ is isomorphic to $Z[1]\oplus Cone(f)$. Being isomorphic to the weakly contractible complex $Cone(\varphi)$, $Z[1]\oplus Cone(f)$ is weakly contractible itself. The same holds for its direct factor $Cone(f)$. Now by definition, the mapping cone $Cone(f)$ of the chain map $f$ of ind-complexes is weakly contractible iff $f$ satisfies criterion (5.13). 
\end{proof}
\\
\\
This criterion plays a crucial role in the proof of the following theorem, which may be viewed as the main feature which distinguishes local cyclic cohomoloy from the other cyclic theories.

\begin{theorem} (Limit Theorem, \cite{Pu4}, 3.15, 6.16)
Let $"\underset{\underset{\lambda\in\Lambda}{\longrightarrow}}{\lim}"B_\lambda$ be a countable directed family of Banach algebras and let 
$$
f\,=\,\underset{\longleftarrow}{\lim}f_\lambda:"\underset{\underset{\lambda\in\Lambda}{\longrightarrow}}{\lim}"B_\lambda\,\longrightarrow B
$$
 be a homomorphism to another Banach algebra. 
Suppose that 
\begin{itemize}
\item $B$ is separable and possesses the Grothendieck approximation property.
\item The image $Im(f)\,=\,\underset{\underset{\lambda\in\Lambda}{\longrightarrow}}{\lim}f_\lambda(B_\lambda)$ is dense in $B$.
\item There exists a constant $C$ such that 
$$
\underset{\underset{\lambda\in\Lambda}{\longrightarrow}}{\overline{\lim}}\parallel b_\lambda\parallel_\lambda\leq C\parallel f(b)\parallel_B
$$
for all 
$$
b=\underset{\underset{\lambda\in\Lambda}{\longrightarrow}}{\lim}\,b_\lambda\,\in\,\underset{\underset{\lambda\in\Lambda}{\longrightarrow}}{\lim}\,B_\lambda.
$$
\end{itemize}
Then $f$ is a local $HC$-equivalence
\end{theorem}

\newpage

\section{The Chern-Connes character of the $\gamma$-element}

The aim of this chapter is to establish
\begin{theorem}
Let $\Gamma$ be a word-hyperbolic group, let $\gamma\in KK_\Gamma(\C,\C)$ be Kasparov's $\gamma$-element, and let 
$j_r(\gamma)\in KK(C_r^*\Gamma,C_r^*\Gamma)$ be its image under "descent" \cite{Ka2}. Then
$$
\begin{array}{ccccc}
ch_{biv}(j_r(\gamma)) & = & 1 & \in & HC_{loc}(C_r^*\Gamma,\,C_r^*\Gamma).
\end{array}
\eqno(6.1)
$$
\end{theorem}
The theorem will follow from a series of intermediate results.
\begin{prop}(\cite{Jol}, \cite{La3})

Let $(\Gamma,S)$ be a word hyperbolic group. There exists an unconditional Banach algebra $A\Gamma$ over $\gamma$, which is a smooth subalgebra of 
$C^*_r\Gamma$ in the sense of \cite{Pu4},3.8. In particular, $A\Gamma$ is dense and closed under holomorphic functional calculus in $C^*_r\Gamma$ and the inclusion $i:\,A\Gamma\,\to\,C^*_r\Gamma$ is a local $HC$-equivalence.
\end{prop}

\begin{proof}
According to Jolissaint \cite{Jol} (see also \cite{La3}), 
the completion ${\mathfrak A}_k\Gamma$ of the complex group ring $\C\Gamma$ with respect to the norm
$$
\parallel \sum a_gu_g\parallel^2_k\,=\,\sum (1+\ell_S(g))^{2k}\vert a_g\vert^2
\eqno(6.2)
$$
is an unconditional Banach algebra over $\Gamma$, provided that $k\in\N$ is large enough. Fix such an integer $k_0$ and put $A\Gamma={\mathfrak A}_{k_0}\Gamma$. The intersection ${\mathfrak A}_\infty\Gamma=\underset{k}{\bigcap}\,{\mathfrak A}_k\Gamma$ is a smooth Fr\'echet subalgebra \cite{Pu4}, 3.8. of $C_r^*\Gamma$ by \cite{Pu3}, 4.2. The same reasoning shows that ${\mathfrak A}_\infty\Gamma$ is a smooth subalgebra of ${\mathfrak A}_k\Gamma$ provided that $k$ is large enough. While it is obvious that an unconditional Banach algebra has the Grothendieck approximation property, the same assertion for the reduced group $C^*$-algebra of a word-hyperbolic group is a deep fact which follows for example from \cite{Oz}.  
One may therefore apply the smooth subalgebra theorem \cite{Pu4}, 5.15, and conclude that the inclusions 
${\mathfrak A}_\infty\Gamma\hookrightarrow C_r^*\Gamma$ and 
${\mathfrak A}_\infty\Gamma\hookrightarrow A\Gamma$ are local $HC$ equivalences. The same holds then also for the inclusion $A\Gamma\hookrightarrow C_r^*\Gamma$. Alternatively, one may adapt the argument of the proof of Proposition 1.2 in \cite{La3} to deduce directly
that $A\Gamma$ is a smooth subalgebra of $C_r^*\Gamma$, and applies then \cite{Oz} and \cite{Pu4}, 5.15.
\end{proof}

The limit theorem allows to conclude from the previous proposition that 
the morphism of ind-algebras
$$
"\underset{\lambda\searrow 1}{\lim}" A_\lambda\Gamma\,\longrightarrow\,C_r^*\Gamma
\eqno(6.3) 
$$
is a local $HC$-equivalence.\\
\\
\begin{lemma}
Let $\lambda\mapsto R_\lambda$ be a monotone increasing map from $]1,\infty[$ to itself. Then the canonical morphism of ind-algebras
$$
"\underset{\lambda\searrow 1}{\lim}" q_{R_\lambda}(A_\lambda\Gamma)\,\longrightarrow\,q(C_r^*\Gamma)
\eqno(6.4) 
$$
is a local $HC$-equivalence.
\end{lemma}

\begin{proof}
It follows from a well known argument of J.~Cuntz (see \cite{Cu1}) and the excision theorem in analytic and local cyclic cohomology 
\cite{Pu2} that the vertical arrows in the commutative diagram
$$
\begin{array}{ccccc}
& "\underset{\lambda\searrow 1}{\lim}" q_{R_\lambda}(A_\lambda\Gamma) & \longrightarrow & q(C_r^*\Gamma) & \\
& & & & \\
q_{R_\lambda}(id,0) & \downarrow & & \downarrow & q(id,0) \\
 & & & & \\
& "\underset{\lambda\searrow 1}{\lim}" A_\lambda\Gamma & \longrightarrow & C_r^*\Gamma & \\
\end{array}
\eqno(6.5)
$$
are local (and in fact analytic) $HC$-equivalences. The claim follows then from (6.3).
\end{proof}

\begin{lemma}
Let $p\in{\cal K(H)}$ be a projection of rank one.
The morphisms in the commutative diagram of ind-algebras
$$
\begin{array}{ccc}
 "\underset{\lambda\searrow 1}{\lim}" A_\lambda(\Gamma, {\cal K(H)}) & \longrightarrow &
C_r^*\Gamma\otimes_{C^*}{\cal K(H)}  \\
 & &    \\
i_p \uparrow & & \uparrow  i_p  \\
 & &    \\
 "\underset{\lambda\searrow 1}{\lim}" A_\lambda\Gamma & \longrightarrow &
C_r^*\Gamma   \\
\end{array}
\eqno(6.6)
$$
are local $HC$-equivalences.
\end{lemma}

\begin{proof}
The upper horizontal arrow is well defined by \cite{Pu3}, 4.4.
Let $B$ be a Banach algebra and let $i_p:\C\to{\cal K}$ be a homomorphism which sends $1\in\C$ to a rank one projection $p\in{\cal K(H)}$. (All these homomorphisms are conjugate to each other.)
Let $p_n,n\in\N,$ be an increasing sequence of finite rank projections such that
$\underset{n\to\infty}{\lim}p_n{\cal K}p_n\,\simeq\,\underset{n\to\infty}{\lim}M_n\C$ is dense in $\cal K$.
By matrix stability of cyclic homology
$$
B\overset{i_{p}}{\longrightarrow}"\underset{n\to\infty}{\lim}" M_nB
\eqno(6.7)
$$
is a local $HC$-equivalence. The limit theorem 5.2 implies then that 
$$
i_{p}:C\,\longrightarrow\,"\underset{n\to\infty}{\lim}"M_nC\to
C\otimes_{C^*}{\cal K(H)},
\eqno(6.8)
$$
$C$ a $C^*$-algebra, and
$$
i_{p}:A\Gamma\,\longrightarrow\,"\underset{n\to\infty}{\lim}" M_n A\Gamma\to
A(\Gamma,{\cal K(H)}),
\eqno(6.9)
$$
$A\Gamma$ an unconditional Banach algebra over $\Gamma$,
are local $HC$-equivalences. The same holds for formal inductive limits of such algebras. Thus the vertical arrows in the diagram above are local $HC$-equivalences and the same holds by (6.3) for the horizontal arrows.
\end{proof}

We have now a look at the various characteristic homomorphisms attached to 
the Kasparov bimodule $\gamma_{KS}$ \cite{KaSk} representing Kasparov's Gamma- element.

\begin{prop}
Let $\Gamma$ be a word-hyperbolic group.
Let $\gamma_{KS}$ \cite{KaSk} be the Kasparov bimodule representing Kasparov's Gamma-element and let $j_r(\gamma_{KS})$ be its image under descent \cite{Ka2}. Let 
$$
\Phi_{j_r(\gamma_{KS})}:\,q(C_r^*\Gamma)\,\longrightarrow\,C_r^*\Gamma\otimes_{C^*}{\cal K(H)}
\eqno(6.10)
$$ 
be the assocciated characteristic homomorphism of $C^*$-algebras 
and let for $\lambda\mapsto R_\lambda$ sufficiently large 
$$
\Psi_{\gamma_{KS}}:
"\underset{\lambda\searrow 1}{\lim}" q_{R_\lambda}(A_\lambda\Gamma)\, \longrightarrow\,
"\underset{\lambda\searrow 1}{\lim}" A_\lambda(\Gamma, {\cal K(H)}) 
\eqno(6.11)
$$ 
be the corresponding homomorphism of unconditional ind-Banach algebras.
Then $\Phi_{j_r(\gamma_{KS})}$ is a local $HC$-equivalence if and only if $\Psi_{\gamma_{KS}}$ is.
\end{prop}

\begin{proof}
Working again with the augmented instead of the ordinary Rips complex we may assume as in remark 3.4 that the basic operator $F$ of the equivariant Fredholm module (Fredholm representation) 
$$
\gamma_{KS}=({\cal H},\varphi,\rho,F)
\eqno(6.12)
$$ 
satisfies $F^2=1$. The same holds then for the Fredholm module
$$
j_r(\Gamma)=({\cal H}\widehat{\rtimes}_\rho C_r^*\Gamma,\varphi\widehat{\otimes} id,F\widehat{\otimes}id)
\eqno(6.13)
$$ 
obtained by descent (see \cite{Ka1}). The associated homomorphisms (2.4) and (2.5) are given by 
$$
\phi_0: C^*_r\Gamma\to{\cal L}({\cal H}\widehat{\rtimes}_\rho C_r^*\Gamma)={\cal L(H)}\rtimes_r\Gamma,\,u_g\mapsto \frac{1+\epsilon}{2}\varphi(1)\frac{1+\epsilon}{2}\otimes u_g
\eqno(6.14)
$$ 
and
$$
\phi_1: C^*_r\Gamma\to{\cal L(H)}\rtimes_r\Gamma,\,u_g\mapsto 
F\frac{1-\epsilon}{2}\varphi(1)\frac{1-\epsilon}{2}\rho(g)F\rho(g)^{-1}\otimes u_g.
\eqno(6.15)
$$
The action of $\Gamma$ on ${\cal L(H)}$ being inner, there is a canonical isomorphism
$$
{\cal L}({\cal H}\widehat{\rtimes}_\rho C_r^*\Gamma)\,=\,{\cal L(H)}\rtimes_r\Gamma\,\overset{\simeq}{\longrightarrow}\,C^*_r\Gamma\otimes_{min}{\cal L(H)},\,T\otimes u_g\mapsto T\rho(g)\otimes u_g.
\eqno(6.16)
$$
Consequently the diagrams
$$
\begin{array}{ccccc}
C^*_r\Gamma & \overset{\phi_i}{\longrightarrow} & {\cal L(H)}\rtimes_r\Gamma
 & \overset{\simeq}{\longrightarrow} & C^*_r\Gamma\otimes_{min}{\cal L(H)} \\
 & & & & \\
\uparrow & & & & \uparrow \\
 & & & & \\
"\underset{\lambda\searrow 1}{\lim}" A_\lambda\Gamma & &  \underset{\psi_i}{\longrightarrow} & & 
"\underset{\lambda\searrow 1}{\lim}" A_\lambda(\Gamma, {\cal L(H)}), \\
\end{array}
\eqno(6.17)
$$
i=0,1, and
$$
\begin{array}{ccccc}
q(C^*_r\Gamma) & \overset{q(\phi_0,\phi_1)}{\longrightarrow} & {\cal K(H)}\rtimes_r\Gamma
 & \overset{\simeq}{\longrightarrow} & C^*_r\Gamma\otimes_{C^*}{\cal K(H)} \\
 & & & &  \\
\uparrow & & & &  \uparrow \\
 & & & & \\
"\underset{\lambda\searrow 1}{\lim}" q_{R_\lambda}(A_\lambda\Gamma) & &  \underset{\Psi_{\gamma_{KS}}}{\longrightarrow} & & 
"\underset{\lambda\searrow 1}{\lim}" A_\lambda(\Gamma, {\cal K(H)}), \\
\end{array}
\eqno(6.18)
$$
commute. The composition of the upper horizontal arrows equals the characteristic homomorphism $\Phi_{j_r(\gamma_{KS})}$. 
According to lemma 6.3 and lemma 6.4 the vertical arrows are local $HC$-equivalences. The assertion follows.
\end{proof}

\begin{prop}
The characteristic morphism 
$$
\Psi_{\gamma_{KS}}:
"\underset{\lambda\searrow 1}{\lim}" q_{R_\lambda}(A_\lambda\Gamma)\, \longrightarrow\,
"\underset{\lambda\searrow 1}{\lim}" A_\lambda(\Gamma, {\cal K(H)}) 
\eqno(6.19)
$$ 
is a local $HC$-equivalence.
\end{prop}

\begin{proof}
We make first the following observation. Let 
$$
f_0,\,f_1:
"\underset{i\in I}{\lim}"C_*^{(i)}\to"\underset{j\in J}{\lim}"E_*^{(j)}
\eqno(6.20)
$$ 
be morphisms of ind-complexes and suppose that for all $i\in I$ there exists $j\in J$ and maps $f_0^{ij},f_1^{ij}:C_*^{(i)}\to E_*^{(j)}$ representing $f_0$ resp. $f_1$, and such that $f_0^{ij}$ is chain homotopic to $f_1^{ij}$. 
(We do not require any compatibility between the homotopies under the structure maps of the inductive systems.) Then the isomorphisms criterion 5.1 
applies to $f_0$ if and only if it applies to $f_1$. 
Recall that the Cartan homotopy formula in cyclic homology \cite{CQ2} 
shows that piecewise $C^1$-homotopic homomorphisms of Banach algebras 
induce chain homotopic maps of analytic cyclic chain complexes.
Together with the previous observation this shows the following.
Let $f,g:"\underset{i\in I}{\lim}"A_i\to"\underset{j\in J}{\lim}"B_j$ be homomorphisms of ind-Banach algebras and suppose that for all $i\in I$ there exists $j\in J$ and homomorphisms $f_{ij},g_{ij}:A_i\to B_j$ representing $f$ resp. $g$, and such that $f_{ij}$ and $g_{ij}$ are piecewise $C^1$-homotopic. Then $f$ is a local $HC$-equivalence if and only if $g$ is.
Theorem 4.5 shows that we may apply this argument to the homomorphisms
$$
\Psi_{\gamma_{KS}}:
"\underset{\lambda\searrow 1}{\lim}" q_{R_\lambda}(A_\lambda\Gamma)\, \longrightarrow\,
"\underset{\lambda\searrow 1}{\lim}" A_\lambda(\Gamma, {\cal K(H)})
\eqno(6.21)
$$ 
and 
$$
"\underset{\lambda\searrow 1}{\lim}" q_{R_\lambda}(A_\lambda\Gamma)\, \overset{q(id,0)}{\longrightarrow}\,"\underset{\lambda\searrow 1}{\lim}" (A_\lambda\Gamma)\,\overset{i_p}{\longrightarrow}\,
"\underset{\lambda\searrow 1}{\lim}" A_\lambda(\Gamma, {\cal K(H)}) 
\eqno(6.22)
$$
for $p\in{\cal K(H)}$ a rank one projection and $\lambda\mapsto R_\lambda$ suitably chosen. The latter one is a local $HC$-equivalenvce by (6.5) and (6.6) and the conclusion follows. We remark that it is tempting to conclude directly that the two morphisms above are equal as morphisms in the derived ind-category. Due to the presence of higher inverse limits one has to be very cautious with such "conclusions" however and may justify them only by a more detailed analysis. Therefore we prefered to give the more indirect argument above.
\end{proof}

\newpage

\begin{proof1}
With the previous results at hand it is easy to proceed.
By construction (see 2.12), the Chern-Connes character of the "Gamma" element of a word-hyperbolic group is given by 
$$
\begin{array}{ccccc}
ch_{biv}(j_r(\gamma)) & = &   i_{p*}^{-1}\circ\Phi_{j_r(\gamma_{KS})*}\circ q(id,0)_*^{-1} & \in & HC^0_{loc}(C^*_r\Gamma,C^*_r\Gamma) \\
\end{array}
\eqno(6.23)
$$
According to 6.5 and 6.6 the characteristic homomorphism $\Phi_{j_r(\gamma_{KS})}$ is a local $HC$-equivalence. Therefore the three 
morphisms $i_{p*}^{-1},\,\Phi_{j_r(\gamma_{KS})*},$ and $q(id,0)_*^{-1}$ as well as their composition $ch_{biv}(j_r(\gamma))$ are isomorphisms in the derived ind-category $ind\,{\cal D}$. 

The "Gamma"-element is an idempotent under the Kasparov product. The multiplicativity of the descent transformation and of the bivariant Chern-Connes character imply that $ch_{biv}(j_r(\gamma))$
is an idempotent isomorphism in $ind\,{\cal D}$. The unique idempotent isomorphism being the identity we conclude 
$$
ch_{biv}(j_r(\gamma))\,=\,1\,\in\,HC^0_{loc}(C_r^*\Gamma,C_r^*\Gamma).
\eqno(6.24)
$$
\end{proof1}

We note also the following consequence of the previous theorem.

\begin{theorem}
Let $\Gamma$ be a word-hyperbolic group and denote by $\Gamma_{tors}$ the subset of elements of finite order. (The ambient group $\Gamma$ acts on $\Gamma_{tors}$ by conjugation). Let ${\cal A}\Gamma$ a sufficiently large unconditional Banach algebra over $\Gamma$ (see \cite{Pu3}) and let $A,B$ be any Banach algebras. Then there are natural isomorphisms
$$
\begin{array}{ccc}
HC_{loc}^*(C_r^*\Gamma\otimes_{\pi} A,B) & \simeq & H^*(\Gamma,\C\Gamma_{tors})\otimes HC_{loc}^*(A,B), 
\end{array}
\eqno(6.25)
$$
and
$$
\begin{array}{ccc}
HC_{loc}^*(A,C_r^*\Gamma\otimes_{\pi}B) & \simeq & H_*(\Gamma,\C\Gamma_{tors})\otimes HC_{loc}^*(A,B).
\end{array}
\eqno(6.26)
$$
\end{theorem}

\section{The Chern-Connes is not rationally injective}

We recall the very elegant argument of G. Skandalis which shows that 
the reduced\\ $C^*$-algebra of a word-hyperbolic group with property $(T)$ 
is not $KK$-equivalent to a nuclear $C^*$-algebra. We do this in order to verify that the result holds even rationally, i.e. with $KK$-theory replaced by $KK\otimes\Q$.

\begin{theorem} (Skandalis, \cite{Sk})
Let $\Gamma$ be a word-hyperbolic group with Kazhdan's Property $(T)$ and let 
$C_r^*\Gamma$ be its reduced group $C^*$-algebra. Let 
$$
C_r^*\Gamma\otimes_{max} C_r^*\Gamma\,\longrightarrow\,C_r^*\Gamma\otimes_{min} C_r^*\Gamma
\eqno(7.1)
$$
be the canonical homomorphism from the maximal to the minimal $C^*$-tensor product. Then the induced map
$$
K_*(C_r^*\Gamma\otimes_{max} C_r^*\Gamma)\otimes\Q\,\longrightarrow\,K_*(C_r^*\Gamma\otimes_{min} C_r^*\Gamma)\otimes\Q
\eqno(7.2)
$$
is not an isomorphism.
\end{theorem}

\begin{proof2}
For every discrete group $\Gamma$ there exists a commutative diagram 
$$
\begin{array}{ccccccccc}
& & & & C^*_{max}\Gamma &  & & & \\

& & & & \downarrow\Delta & & & & \\

0 & \to & J & \to & C_r^*\Gamma\otimes_{max} C_r^*\Gamma &  \to & 
C_r^*\Gamma\otimes_{min} C_r^*\Gamma\,=\,C_r^*(\Gamma\times\Gamma) & \to & 0 \\

& & & & \downarrow\rho & & & & \\

& & & & {\cal L}(\ell^2(\Gamma)) & & & & \\
\end{array}
\eqno(7.3)
$$
with exact row, where $\Delta$ denotes the diagonal and $\rho=\rho_l\otimes\rho_r$ the biregular representation . If $\Gamma$ is word-hyperbolic, the lower left part 
of the diagram fits, according to Akeman-Ostrand, into a commutative square
$$
\begin{array}{ccc}
J & \to & C_r^*\Gamma\otimes_{max} C_r^*\Gamma \\
 & & \\
i\downarrow & & \downarrow\rho \\
 & & \\
{\cal K}(\ell^2(\Gamma)) & \underset{j}{\to} & {\cal L}(\ell^2(\Gamma)) 
\end{array}
\eqno(7.4)
$$
Suppose now that $\Gamma$ has Kazhdan's Property $(T)$. This means that there exists a projector $e\in C_{max}^*\Gamma$ which is mapped under any involutive representation of $C_{max}^*\Gamma$
on a Hilbert space $\cal H$ to the orthogonal projection of $\cal H$ onto the vectors fixed by the corresponding unitary representation of $\Gamma$.
Then on the one hand the image of $\Delta(e)$ in $C_r^*\Gamma\otimes_{min} C_r^*\Gamma$ is zero (and thus $\Delta(e)\in J$), while on the other hand 
$\rho(\Delta(e))\neq 0\in {\cal L}(\ell^2(\Gamma))$, because there are always nonzero vectors fixed under the adjoint representation. Thus $j(i(\Delta(e)))\neq 0$ and $i(\Delta(e))$ is a non-zero projection in ${\cal K}(\ell^2(\Gamma))$. Its class in rational $K$-theory does not vanish.
So the homomorphism $K(i):K(J)\otimes\Q\to K({\cal K}(\ell^2(\Gamma)))\otimes\Q$ is not zero. The long exact sequence in $K$-theory implies then that $K_*(C_r^*\Gamma\otimes_{max} C_r^*\Gamma)\otimes\Q\,\to\,K_*(C_r^*\Gamma\otimes_{min} C_r^*\Gamma)\otimes\Q$ cannot be an isomorphism.
\end{proof2}

\begin{cor} (Skandalis, \cite{Sk})
Let $\Gamma$ be a word-hyperbolic group with Kazhdan's Property $(T)$ and let 
$\gamma\in KK_\Gamma(\C,\C)$ be Kasparov's "Gamma"-element. Then
$$
j_r(\gamma)\,\neq\,1\,\in\,KK(C_r^*\Gamma,C_r^*\Gamma)\otimes\Q
\eqno(7.5)
$$
\end{cor}

\begin{proof2}
By construction (see \cite{Ka2}) there exists a proper nuclear $\Gamma-C^*$-algebra $A$ and elements 
$$
\begin{array}{cc}
\alpha\in KK_\Gamma(A,\C), & \beta\in KK_\Gamma(\C,A)
\end{array}
\eqno(7.6)
$$
such that 
$$
\begin{array}{cc}
\alpha\circ\beta=1\,\in KK_\Gamma(A,A), & \beta\circ\alpha=\gamma\,\in KK_\Gamma(\C,\C).
\end{array}
\eqno(7.7)
$$
Suppose that $j_r(\gamma)=1\,\in\,KK(C_r^*\Gamma,C_r^*\Gamma)\otimes\Q$. Then the horizontal arrows in the commutative diagram 
$$
\begin{array}{ccc}
K(A\rtimes_r\Gamma\otimes_{max}A\rtimes_r\Gamma)\otimes\Q & \overset{j_r(\alpha)\otimes_{max}j_r(\alpha)}{\longrightarrow} &
K(C^*_r\Gamma\otimes_{max}C^*_r\Gamma)\otimes\Q \\
 & & \\
\downarrow & & \downarrow \\
 & & \\
K(A\rtimes_r\Gamma\otimes_{min}A\rtimes_r\Gamma)\otimes\Q & 
\underset{j_r(\alpha)\otimes_{min}j_r(\alpha)}{\longrightarrow} &
K(C^*_r\Gamma\otimes_{min}C^*_r\Gamma)\otimes\Q
\end{array}
\eqno(7.8)
$$
are isomorphisms. The left vertical arrow is an isomorphisms because the algebra $A\rtimes_r\Gamma$ is nuclear ($A$ being nuclear and the $\Gamma$-action being proper) and one concludes that the right vertical arrow is an isomorphism which contradicts theorem 7.1.
\end{proof2}
Now we can state the main result of this paper.
\begin{theorem}
Let $\Gamma$ be a word-hyperbolic group which has Kazhdan's Property $(T)$.
(Such groups are in a suitable sense generic among all word-hyperbolic groups). Then the bivariant Chern-Connes character 
$$
\begin{array}{cccc}
ch_{biv}: & KK(C^*_r\Gamma,C^*_r\Gamma) & \longrightarrow & 
HC_{loc}(C^*_r\Gamma,C^*_r\Gamma)
\end{array}
\eqno(7.9)
$$
is not rationally injective.
\end{theorem}

\begin{proof}
Let $\Gamma$ be a word-hyperbolic group which has Kazhdan's Property $(T)$. 
Let $\gamma\in KK_\Gamma(\C,\C)$ be Kasparov's "Gamma"-element and let 
$j_r(\gamma)\in KK(C^*_r\Gamma,C^*_r\Gamma)$ be the element obtained from it by descent \cite{Ka2}. Then 
$$
j_r(\gamma)\,\neq\,1\,\in\,KK(C^*_r\Gamma,C^*_r\Gamma)\otimes\Q
\eqno(7.10)
$$
by Skandalis' theorem 7.2 while 
$$
ch_{biv}(j_r(\gamma))\,=\,1\,=\,ch_{biv}(1)\,\in\,HC_{loc}(C^*_r\Gamma,C^*_r\Gamma)
\eqno(7.11)
$$
according to theorem 6.1.

\end{proof}

\end{document}